\documentclass[letterpaper,10pt]{article}

\usepackage[latin1]{inputenc}
\usepackage[T1]{fontenc}
\usepackage{enumerate}
\usepackage[all]{xy}
\usepackage{amsmath,amsthm,amsfonts,amssymb}
\usepackage[pdftex,bookmarks=false]{hyperref}

\newtheorem{lem}{Lemma}

\newtheorem{thm}[lem]{Theorem}
\newtheorem*{thm*}{Theorem}
\newtheorem*{thmfr*}{Théorème}
\newtheorem*{thmzp*}{$\mathbf{Z^*_p}$-theorem}
\newtheorem*{thmfrz*}{$\mathbf{Z^*}$-théorème}

\theoremstyle{definition}
\newtheorem{defn}[lem]{Definition}
\newtheorem{rmk}[lem]{Remark}

\newcommand\catname[1]{\mathbf{#1}}

\newcommand\catob{\catname{Ob}}

\newcommand\catmod[1]{{\null_{#1}\catname{Mod}}}

\newcommand\catperm[1]{{\null_{#1}\catname{Perm}}}

\newcommand\catm[1]{{\null_{#1}\catname{M}}}

\newcommand\catf{\catname{F}}
\newcommand\catbr{\catname{Br}}
\newcommand\catfr{\catname{Fr}}

\DeclareMathOperator{\Homom}{Hom}

\DeclareMathOperator{\Endom}{End}
\DeclareMathOperator{\Res}{Res}

\DeclareMathOperator{\Ind}{Ind}
\DeclareMathOperator{\Br}{Br}
\DeclareMathOperator{\br}{br}

\DeclareMathOperator{\Tr}{Tr}

\DeclareMathOperator{\Sl}{Sl}

\newcommand\isom{\tilde{\,\to}\,}
\newcommand\longisom{\tilde{\,\longrightarrow}\,}

\newcommand\ringO{\mathcal O}

\newcommand\pprime{{\prime\prime}} 
\newcommand\dade{\mathcal D}

\title{Brauer-friendly modules and slash functors}
\author{Erwan Biland}
\date{\today}

\begin{document}

\maketitle

\begin{abstract}
This paper introduces the notion of \emph{Brauer-friendly modules}, a generalisation of endo-$p$-permutation modules. A module over a block algebra $\ringO Ge$ is said to be Brauer-friendly if it is a direct sum of indecomposable modules with compatible fusion-stable endopermutation sources. We obtain, for these modules, a functorial version of Dade's slash construction, also known as deflation-restriction. We prove that our \emph{slash functors}, defined over \emph{Brauer-friendly categories}, share most of the very useful properties that are satisfied by the Brauer functor over the category of $p$-permutation $\ringO Ge$-modules. In particular, we give a parametrisation of indecomposable Brauer-friendly modules, which opens the way to a complete classification whenever the fusion-stable sources are classified. Those tools have been used to prove the existence of a stable equivalence between non-principal blocks in the context of a minimal counter-example to the odd $Z^*_p$-theorem.
\end{abstract}

\section*{Introduction}

The $p$-permutation modules, or trivial-source modules, are an essential tool of the local representation theory of finite groups. What makes them so useful is the existence of a \emph{localising} tool, the Brauer functor, which satisfies nice properties such as transitivity, and provides a simple classification of indecomposable $p$-permutation modules, as proven by Broué in \cite[Theorem 3.2~(3)]{Broué1985}. In particular, the properties of the Brauer functor have enabled Rickard, in \cite{Rickard1996}, to prove that a derived equivalence defined by a so-called \emph{splendid} complex of bimodules induces a collection of local derived equivalences.

A possible generalisation of $p$-permutation modules is the notion of endo-$p$-permutation modules, as described by Urfer in \cite{Urfer2007}. Let $\ringO$ be a discrete valuation ring with positive residual characteristic $p$. An indecomposable $\ringO G$-module $M$ is an endo-$p$-permutation module if its source $V$ with respect to a given vertex $P$ is an endopermutation $\ringO P$-module that is fusion-stable with respect to the fusion system $\catf_G(P)$ defined by the group $G$ on the $p$-subgroup~$P$. This characterisation suggests that endo-$p$-permutation modules are not the most natural notion when one wants to take blocks into account. Indeed, a non-principal block $e$ of the group $G$ usually defines a fusion system on a $p$-subgroup $P$ that is finer than the fusion system $\catf_G(P)$ defined by the group $G$. This is why we introduce a more general notion. We say that an $\ringO Ge$-module $M$ is Brauer-friendly if it is a direct sum of indecomposable modules with endopermutation sources that are fusion-stable with respect to the block $e$, as defined in \cite{LinckelmannMazza}, and moreover compatible with one another.

Although they do not seem to have been explicitly described before, these modules are ubiquitous in modular representation theory. Indeed, let $M$ be an indecomposable bimodule that induces a Morita or stable equivalence between two block algebras. If $M$ admits a diagonal vertex, it follows from \cite[\S 7]{Puig1999} that a source of $M$ must be a fusion-stable endopermutation module, \emph{i.e.}, $M$ must be a Brauer-friendly module. Puig usually studies Brauer-friendly modules by turning them into endopermutation modules over source algebras.
However, it seems reasonable to study Brauer-friendly modules directly at the level of block algebras, a more common framework in modular representation theory. Such is the aim of the present paper. 

In Sections \ref{sec:prelim} and \ref{sec:VertexSubGreenCorr}, we review the theory of subpairs that has been defined in \cite{AlperinBroué}. Then we refine the definition of the Brauer functor, as well as Green's theory of vertices and sources (following \cite{Sibley1990}) to take subpairs into consideration.
In Section \ref{sec:BFModules}, we define the notion of a Brauer-friendly module. We prove that a module $M$ over a block algebra $\ringO Ge$ with defect group $D$ is Brauer-friendly if, and only if, the corresponding module over a source algebra $A$ of $\ringO Ge$ is an endopermutation $\ringO D$-module.

In Sections \ref{sec:SlashFunc} and \ref{sec:Slash}, we prove that Dade's localising tool, known as the slash construction or deflation-restriction, can be applied to Brauer-friendly modules. We also show that this construction can be turned into a functor, provided that one focuses on what we call a Brauer-friendly category, \emph{i.e.}, a category of compatible Brauer-friendly modules. This functoriality may have important consequences in local representation theory. For instance, by allowing one to \emph{localise} a complex of compatible Brauer-friendly modules, it would make it easier to deal with the local properties of a derived equivalence between block algebras, even though this equivalence is not defined by a splendid complex.

In Section \ref{sec:ParamIndBF}, we use the properties of slash functors to give a practical parametrisation of indecomposable Brauer-friendly modules. As mentioned above, a Morita or stable equivalence between block algebras is often defined by such a module, so our parametrisation can be an important tool when it comes to finding a bimodule that induces such an equivalence.

Let us give a more precise example of application. The search for a modular proof of the odd $Z^*_p$-theorem is an important open question of the representation theory of finite groups. In the context of a (putative) minimal counter-example to that theorem, thanks to the above tools, we have been able to consider a family of local Morita equivalences and \emph{glue} them together to obtain a stable equivalence between non-principal block algebras, as appears in \cite[Chapter 4]{Biland2013}.

\section{Subpairs and the Brauer functor}
\label{sec:prelim}

We first give a few notations that we use throughout this article. We let $\ringO$ be a complete discrete valuation ring with algebraically closed residue field $k$ of characteristic $p$. This includes the case $\ringO=k$, so that every result that is proven over the ring $\ringO$ remains true over the field~$k$. 

For any finite group $G$, we denote by $\Delta G =\{(g,g); g\in G\}$ the diagonal subgroup of the direct product $G\times G$. We denote by $\catmod{\ringO G}$ the category of $\ringO G$-modules and morphisms of $\ringO G$-modules, and by $\catperm{\ringO G}$ the full subcategory of $p$-permutation $\ringO G$-modules, \emph{i.e.}, direct summands of $\ringO G$-modules that are $\ringO$-free and admit a $G$-stable basis.
For an element $g\in G$ and an object $X$, the notation $^gX$ stands for the object $gXg^{-1}$ whenever this makes sense. For any element $x$ of the group algebra $\ringO G$, we denote by $\bar x$ its image by the natural projection map $\ringO G\twoheadrightarrow kG$. For any two groups $G$ and $H$, an $(\ringO G,\ringO H)$-bimodule $M$ will be considered as an $\ringO (G\times H)$-module by setting
\[
(g, h)\cdot m = g\cdot m\cdot h^{-1}
\qquad
\text{for any $g\in G$, $h\in H$ and $m\in M$.}
\]

We recall the definition of subpairs from \cite{AlperinBroué}.
Let $G$ be a finite group, and $e$ be a block of the algebra $\ringO G$. 
A subpair of the group $G$ is a pair $(P,e_P)$, where $P$ is a $p$-subgroup of $G$ and $e_P$ is a block of the group algebra $\ringO C_G(P)$. The subpair $(P,e_P)$ is an $e$-subpair if $\bar e_P\br_P(e)\neq0$, where $\br_P:(\ringO G)^P\to kC_G(P)$ denotes the Brauer morphism. The idempotent $e_P$ is a block of the algebra $\ringO H$ whenever $H$ is a local subgroup of $G$ with respect to the subpair $(P,e_P)$, \emph{i.e.}, a subgroup such that $C_G(P)\leqslant H\leqslant N_G(P,e_P)$.

The group $G$ acts by conjugation on the poset of $e$-subpairs of $G$. This action can be described by the Brauer category $\catbr(G,e)$, defined as follows. An object is an $e$-subpair $(P,e_P)$. An arrow $\phi:(P,e_P)\to (Q,e_Q)$ is a group morphism $\phi:P\to Q$ such that there exists an element $g\in G$ that satisfies $^g(P,e_P) \leqslant (Q,e_Q)$ and $\phi(x) = \null^gx$ for any $x\in P$. The composition of arrows is the usual composition of group morphisms.
The Brauer category $\catbr(G,e_0)$ of the principal block $e_0$ actually describes the action of the group $G$ on the poset of its $p$-subgroups, since an $e_0$-subpair $(P,e_P)$ is completely determined by the $p$-subgroup $P$. This is called the Frobenius category of the group $G$, and denoted by $\catfr(G)$.

Let $(D,e_D)$ be a maximal $e$-subpair, \emph{i.e.}, an $e$-subpair such that the $p$-subgroup $D$ of $G$ is a defect group of the block $e$. Let $\catf_{(G,e)}(D,e_D)$ be the full subcategory of $\catbr(G,e)$ of which the objects are the $e$-subpairs contained in $(D,e_D)$. This subcategory is called the fusion system of the block $e$ with respect to the subpair $(D,e_D)$; it is equivalent to the Brauer category $\catbr(G,e)$. A source idempotent of the block $e$ with respect to the maximal subpair $(D,e_D)$ is a primitive idempotent $i$ of the algebra $(\ringO Ge)^D$ such that $\bar e_D\br_D(i)\neq 0$. The $D$-interior algebra $A=i\ringO Gi$ is called a source algebra of the block $e$. The $(A,\ringO Ge)$-bimodule $i\ringO G$ induces a Morita equivalence $A\sim \ringO Ge$. Moreover, the fusion system $\catf_{(G,e)}(D,e_D)$ can be read in the $\ringO(D\times D)$-module $A$. More details on this approach may be found in \cite{LinckelmannUnpublished}.

We now recall and slightly generalise the classical definition of the Brauer functor.
Let $G$ be a finite group and $P$ be a $p$-subgroup of $G$. We write $\bar N_G(P) = N_G(P)/P$. For any $\ringO G$-module $M$, we denote by $\Br_P(M)$ the Brauer quotient of $M$, which some authors denote by $M(P)$, \emph{i.e.}, the $k\bar N_G(P)$-module
\[
\Br_P(M) \ = \ M^P \Bigl/ \Bigl(\sum_{Q<P}\Tr_Q^P(M^Q) +\mathfrak m M^P\Bigr),
\]
where $M^P$ is the submodule of $P$-fixed points in $M$, $\Tr_Q^P:M^Q\to M^P$ is the relative trace map and $\mathfrak m$ is the maximal ideal of the local ring $\ringO$. We denote by $\br_P^M:M^P\to \Br_P(M)$ the projection map. Any morphism of $\ringO P$-modules $u:L\to M$ induces a $k$-linear map $\Br_P(u):\Br_P(L)\to\Br_P(M)$. If $u$ is a morphism of $\ringO G$-modules, then $\Br_P(u)$ is a morphism of $k\bar N_G(P)$-modules. This defines a functor
\[
\Br_P : \catmod{\ringO G} \ \to\ \catmod{k\bar N_G(P)}.
\]
Notice that we write the Brauer functor $\Br_P$ with a capital B, and the Brauer map $\br_P$ with a lowercase b. This tool can be adapted to take subpairs into consideration. Let $e$ be a block of the algebra $\ringO G$, $(P,e_P)$ be an $e$-subpair of the group $G$, and $M$ be an $\ringO Ge$-module. We define the Brauer quotient of $M$ with respect to the subpair $(P,e_P)$ by setting
\[
\Br_{(P,e_P)}(M) \ = \ \Br_P(e_P M).
\]
This generalised Brauer quotient has a natural structure of $k\bar N_G(P,e_P)\bar e_P$-module, where we write $\bar N_G(P,e_P) = N_G(P,e_P)/P$. Notice that we slightly abuse notations by identifying the block $\bar e_P\in kN_G(P,e_P)$ to a block of the quotient algebra $k\bar N_G(P,e_P)$.

Let $L$ and $M$ be two $\ringO Ge$-modules, and let $u:L\to M$ be a morphism of $\ringO P$-modules. Then the map $u$ induces a morphism of $\ringO P$-modules $e_Pue_P:e_PL\to e_P M$.
We set $\Br_{(P,e_P)}(u)=\Br_P(e_Pue_P)$, a $k$-linear map from $\Br_{(P,e_P)}(L)$ to $\Br_{(P,e_P)}(M)$.
If $u:L\to M$ is a morphism of $\ringO Ge$-modules, then $\Br_{(P,e_P)}(u)$ is a morphism of $k\bar N_G(P,e_P)\bar e_P$-modules. This defines a functor
\[
\Br_{(P,e_P)} : \catmod{\ringO Ge} \ \to\ \catmod{k\bar N_G(P,e_P)\bar e_P}.
\]

\section{Vertex subpairs and the Green correspondence}
\label{sec:VertexSubGreenCorr}

In this section, we review the notion of a vertex subpair of an indecomposable module, which has been defined in \cite{Sibley1990}. Let $G$ be a finite group, and $e$ be a block of the algebra $\ringO G$. The following lemma extends Green's theory of vertices and sources, as well as the Green correspondence, to take $e$-subpairs into consideration.

\begin{lem}
\label{lem:vertex subpair}
Let $M$ be an indecomposable $\ringO Ge$-module and $(P,e_P)$ be an $e$-subpair of the group $G$. The following conditions are equivalent.
\vspace{-\smallskipamount}
\begin{enumerate}[(i)]
\item The $p$-subgroup $P$ is contained in a vertex of $M$, and $M$ is isomorphic to a direct summand of the $\ringO Ge$-module $e\ringO Ge_P\otimes_{\ringO P} V$ for some indecomposable $\ringO P$-module $V$.
\item The $\ringO G$-module $M$ is relatively $P$-projective, and the $\ringO P$-module $e_PM$ admits an indecomposable direct summand $V$ with vertex $P$.
\item
The $\ringO N_G(P,e_P)$-module $e_PM$ admits an indecomposable direct summand $L$ with vertex $P$ such that $M$ is isomorphic to a direct summand of the induced module $\Ind_{N_G(P,e_P)}^G L$
\item
The $p$-group $P$ is a vertex of $M$, and the Green correspondent of $M$ with respect to this vertex is an $\ringO N_G(P)$-module $M'$ that belongs to the block $e'_P=\Tr_{N_G(P,e_P)}^{N_G(P)} e_P$ of the algebra $\ringO N_G(P)$.
\end{enumerate}
If these conditions are satisfied, then the $\ringO N_G(P,e_P)$-module $L$ of (iii) and the $\ringO N_G(P)$-module $M'$ of (iv) satisfy the relations
\vspace{-\smallskipamount}
\[
M'\ \simeq\ \Ind_{N_G(P,e_P)}^{N_G(P)} L
\quad\text{ and }\quad
L\ \simeq\ e_P M'.
\]
\end{lem}

\begin{proof}
With the notations of Condition (iv), we know from \cite[Theorem 1.6]{Harris2007} that the restriction/induction functors $e_P \Res^{\,N_G(P)}_{N_G(P,e_P)}$ and $\Ind^{\,N_G(P)}_{N_G(P,e_P)}$ define a Morita equivalence between the block algebras
$\ringO N_G(P,e_P)e_P$ and $\ringO N_G(P)e_P'$.
The equivalence (iii)$\Leftrightarrow$(iv), as well as the last statement of the Lemma, immediately follow from this Morita equivalence and the classical definition (and uniqueness) of the Green correspondent of $M$.

We now suppose that (iii) is satisfied. Then the restriction $\Res^{N_G(P,e_P)}_P L$ admits a direct summand $V$ with vertex $P$, so the $\ringO P$-module $e_PM$ also admits $V$ as a direct summand. Moreover, the induced module $\Ind_{N_G(P,e_P)}^G L$ is relatively $P$-projective, so the $\ringO G$-module $M$ is relatively $P$-projective. Since $V$ is also a direct summand of $M$ and admits $P$ as a vertex, it follows that the $p$-group $P$ is a vertex of $M$. Moreover, the $\ringO P$-module $V$ is a source of $L$. Thus $L$ is isomorphic to a direct summand of the $\ringO N_G(P,e_P)e_P$-module $\ringO N_G(P,e_P)e_P \otimes_{\ringO P} V$, and $M$ to a direct summand of the $\ringO Ge$-module
\[
e\Ind_{N_G(P,e_P)}^G \Bigl( \ringO N_G(P,e_P)e_P \otimes_{\ringO P} V \Bigr)
\ \simeq \ 
e\ringO Ge_P \otimes_{\ringO P} V.
\]
We have proven that the condition in (iii) implies (i) and (ii).

Next, we suppose that (i) is satisfied. The $\ringO G$-module $M$ is relatively $P$-projective and the $p$-group $P$ is contained in a vertex of $M$, so $P$ itself is a vertex of $M$. Moreover, $M$ is isomorphic to a direct summand of the $\ringO Ge$-module 
\[
\ringO Ge_P\otimes_{\ringO P} V
\ \simeq\
\Ind_{N_G(P)}^G \bigl( \ringO N_G(P)e_P\otimes_{\ringO P} V \bigr ),
\]
so there is an indecomposable direct summand $M'$ of the $\ringO N_G(P)e'_P$-module $\ringO N_G(P)e_P\otimes_{\ringO P} V$ such that $M$ is isomorphic to a direct summand of $\Ind_{N_G(P)}^G M'$. The $\ringO N_G(P)$-module $M'$ is therefore a Green correspondent of $M$ with respect to the vertex $P$, so that (i) implies (iv).

Finally, we suppose that (ii) is satisfied. Then the restriction $\Res^G_P M$ admits an indecomposable direct summand with vertex $P$, and the $p$-group $P$ contains a vertex of $M$, so $P$ itself is a vertex of $M$. Let $M'$ be a Green correspondent of $M$ with respect to this vertex, and let $f$ be the block of the algebra $\ringO N_G(P)$ such that $fM'=M'$. Then the uniqueness of the Green correspondent implies that no indecomposable direct summand of the $\ringO P$-module $(1-f)M$ admits $P$ as a vertex. Since $e_PM=fe_PM\oplus (1-f)e_P M$, it follows that $fe_P\neq 0$, so that $f=\Tr_{N_G(P,e_P)}^{N_G(P)} e_P$. Thus (ii) implies (iv), and the four conditions are equivalent. 
\end{proof}

We derive from the above lemma a few definitions and easy consequences that extend some of Green's classical definitions.

Let $M$ be an indecomposable $\ringO Ge$-module. It follows from Nagao's theorem (see, \emph{e.g.}, \cite[Theorem 6.3.1]{Benson1991}) that there exists an $e$-subpair $(P,e_P)$ of the group $G$ that satisfies the statement in Lemma \ref{lem:vertex subpair} (iv), hence the statements in (i), (ii) and (iii). Such an $e$-subpair is called a vertex subpair of the indecomposable module $M$ (this is consistent with \cite[Definition 2.6]{Sibley1990}). A source of $M$ with respect to the vertex subpair $(P,e_P)$ is an $\ringO P$-module $V$ that satisfies any one of the equivalent conditions (i) and (ii). A source triple of $M$ is a triple $(P,e_P,V)$, where $V$ is a source of $M$ with respect to the vertex subpair $(P,e_P)$. The source triples of $M$ form an orbit under the action of the group $G$ by conjugation.

A Green correspondent of $M$ with respect to the vertex subpair $(P,e_P)$ is an $\ringO N_G(P,e_P)$-module $L$ that satisfies the condition in (iii). The mapping $M\mapsto L$ induces a one-to-one correspondence between the isomorphism classes of indecomposable $\ringO Ge$-modules with vertex subpair $(P,e_P)$, and the isomorphism classes of indecomposable $\ringO N_G(P,e_P)e_P$-modules with vertex $P$.

The following properties of vertex subpairs and sources are closer to the approach of \cite{Sibley1990}.
 
\begin{lem}
\label{lem:vertex subpair and Brauer}
Let $M$ be an indecomposable $\ringO Ge$-module and $(P,e_P,V)$ be a source triple of $M$.
\begin{enumerate}[(i)]
\item 
There exists a primitive idempotent $i$ of the algebra $(\ringO G)^P$ such that $\bar e_P\br_P(i)\neq 0$ and that $M$ is isomorphic to a direct summand of the $\ringO Ge$-module $\ringO Gi\otimes_{\ringO P} V$.
\item 
There exists a defect group $D$ of the block $e$ such that $P\leqslant D$, there exists a primitive idempotent $j$ of the algebra $(\ringO G)^D$ such that $\br_D(j)\neq 0$ and $\bar e_P\br_P(j)\neq 0$, and that $M$ is isomorphic to a direct summand of the $\ringO Ge$-module $\ringO Gj\otimes_{\ringO D} \Ind_P^D V$.
\end{enumerate}
\end{lem}

\begin{proof}
By assumption, $M$ is a direct summand of the $\ringO Ge$-module $e\ringO Ge_P\otimes_{\ringO P} V$. Consider a decomposition $ee_P = i_1+\ldots+i_n$ of the idempotent $ee_P$ into mutually orthogonal primitive idempotents in the algebra $(\ringO G)^P$. This brings
\[
e\ringO Ge_P\otimes _{\ringO P} V = (\ringO Gi_1\otimes V) \oplus\cdots\oplus
(\ringO Gi_n\otimes V).
\]
Since the $\ringO Ge$-module $M$ is indecomposable, by the Krull-Schmidt theorem, it must be isomorphic to a direct summand of the $\ringO Ge$-module $\ringO Gi_l\otimes _{\ringO P} V$ for some $l\in\{1,\ldots,n\}$. If $\bar e_P\br_P(i_l)=0$, then the idempotent $i_l$ lies in $\Tr_Q^P((\ringO G)^Q)$ for some proper subgroup $Q$ of $P$, so $M$ is relatively $Q$-projective and cannot admit $P$ as a vertex. This contradiction proves~(i).

Let $\alpha$ be the point of the algebra $(\ringO Ge)^P$ that contains the idempotent $i=i_l$. Since $\br_P(i)\neq 0$, the pointed group $P_\alpha$ is local. Let $D_\beta$ be a defect pointed group of the block algebra $\ringO Ge$ such that $P_\alpha \leqslant D_\beta$ (see \cite[\S 18]{Thévenaz1995}). Then the $p$-group $D$ is a defect group of the block $e$ and there exists an idempotent $j\in \beta$ such that $\br_D(j)\neq 0$ and $ij=ji=i$. Then $\bar e_P\br_P(j)\neq 0$, and the $\ringO Ge$-module $\ringO Gi\otimes_{\ringO P} V$ is a direct summand of
\[
\ringO Gj\otimes_{\ringO P} V
\ \simeq \ 
\ringO Gj\otimes_{\ringO D} \Ind_P^D V.
\]
\vspace{-12mm}

\end{proof}

The following result deals with the behaviour of vertex subpairs and sources with respect to restriction. It is similar to \cite[Corollary 2.7]{Sibley1990}, with an additionnal statement about sources.

\begin{thm}
\label{thm:source triples and restriction}
Let $M$ be an $\ringO Ge$-module, and $H$ be a subgroup of $G$. Let $(Q,e_Q,V)$ be a source triple of an indecomposable direct summand $X$ of the restriction $\Res^G_H M$. Assume that the subgroup $H$ contains the centraliser $C_G(Q)$. Then $(Q,e_Q)$ is an $e$-subpair of the group $G$, and there exists a source triple $(R,e_R,W)$ of an indecomposable direct summand of the $\ringO Ge$-module $M$ such that $(Q,e_Q)\leqslant (R,e_R)$ and that $V$ is a direct summand of the $\ringO Q$-module $\Res^R_Q W$.
\end{thm}

\begin{proof}
We may assume that the $\ringO Ge$-module $M$ is indecomposable. Let $(P,e_P,U)$ be a source triple of $M$. By Lemma \ref{lem:vertex subpair and Brauer}, there is a primitive idempotent $i\in(\ringO Ge)^P$ such that $\bar e_P\br_P(i)\neq 0$ and $X$ is (isomorphic to) an indecomposable direct summand of the $\ringO H$-module $\ringO Gi\otimes_{\ringO P} U$.
Since $C_G(Q)\leqslant H$, the block $e_Q$ of $\ringO C_H(Q)$ is also a block of the algebra $\ringO QC_G(Q)$. The $\ringO QC_G(Q)$-module $e_Q X$ admits an indecomposable direct summand $X'$ with vertex $Q$ and source $V$. Up to replacing $X$ by $X'$, we may now assume that $H=QC_G(Q)$, and that $X$ is an indecomposable direct summand of the $\ringO H$-module $L=e_Q\ringO Gi\otimes_{\ringO P}U$.

We consider the $\ringO H$-module $L_0 = \ringO G\otimes_{\ringO P} U$. The map $f:L_0\to L_0$ defined by $f(g\otimes u) = e_Qgi\otimes u$ is an idempotent of $\Endom_{\ringO H}(L_0)$ such that $f(L_0)=L$, so $L$ is a direct summand of $L_0$. Let $L_0=X_0\oplus\cdots\oplus X_n$ be a Krull-Schmidt decomposition of the $\ringO H$-module $L_0$ such that $X_0=X$, $L=X_0\oplus\cdots\oplus X_m$ and $\ker f=X_{m+1}\oplus\cdots\oplus X_n$ for some $m\in\{1,\ldots,n\}$.

We then consider the decomposition $L_0 = \bigoplus_{g\in \mathcal R} L_g$, where $\mathcal R$ is a set of representatives for the double class set $H\!\setminus\! G/P$, and $L_g=\ringO HgP\otimes_{\ringO P} U$ for any $g\in \mathcal R$.
This can be refined into a Krull-Schmidt decomposition $L_0=Y_0\oplus\cdots\oplus Y_n$ of the $\ringO H$-module $L_0$.
From the proof of the Krull-Schmidt theorem in \cite{Benson1991}, we know that there exists an $l\in\{1,\ldots,n\}$ such that the projection map from $Y_l$ onto $X_0$ along $X_1\oplus\cdots\oplus X_n$ is an isomorphism. We may assume $l=0$. Then $f(Y_0)$ is a complement of the direct summand $X_1\oplus\cdots\oplus X_m$ in the $\ringO H$-module $L$, so the $\ringO H$-modules $f(Y_0)$ and $X_0$ are isomorphic. We may assume $X_0=f(Y_0)$.

By construction, there exists a $g\in\mathcal R$ such that the $\ringO H$-module $Y_0$ is a direct summand of $L_g$. The isomorphism
\[
L_g \ \simeq \ 
\Ind_{H\cap\null^g\!P}^{H} \Res^{^g\!P}_{H\cap\null^g\!P} gU
\]
shows that the $\ringO H$-module $L_g$, hence also $Y_0$, is relatively $H\cap\null^g\!P$-projective. Since the vertex $Q$ of the $\ringO H$-module $Y_0\simeq X_0$ is normal in $H$, it follows that $Q\leqslant\null^g\!P$.
Assume for a moment that $(Q,e_Q)\not\leqslant\null^g(P,e_P)$.
For any $h\in H$, we have
\[
f(hgP\otimes_{\ringO P} U) 
\ =\ 
e_Q(hgP)i\otimes_{\ringO P} U
\ =\ 
e_Q(\,^{hg}i).(hgP)\otimes_{\ringO P} U. 
\]
The $e$-subpair $(Q,e_Q)$ is normalised by the group $H$, so we have $Q\leqslant\null^{hg}P$ but $(Q,e_Q)\not\leqslant\null^{hg}(P,e_P)$. By definition, $i$ is a primitive idempotent of $(\ringO Ge)^P$ such that $\bar e_P\br_P(i)\neq 0$. Thus $^{hg}i$ is a primitive idempotent in $(\ringO Ge)^Q$, and it follows from \cite[Lemma 40.1]{Thévenaz1995} that $\bar e_Q\br_Q(\null^{hg}i)=0$. Then we obtain
$
f(hgP\otimes_{\ringO P} U) \subseteq J_Q L_0,
$ 
where $J_Q\subseteq (\ringO Ge)^Q$ is the kernel of the Brauer map $\br_Q$. This implies
\[
X_0 \ =\ f(Y_0)
\ \subseteq\ 
f(L_g)
\ \subseteq\ 
J_Q\,L_0.
\] 
We know that $X_0$ is a direct summand of the $\ringO H$-module $L_0$. Consider the endomorphism algebra $A=\Endom_\ringO(L_0)$, and let $a\in A^H$ be an idempotent such that $X_0=aL_0$. Then there is an isomorphism $\Endom_\ringO (X_0)\simeq aAa$, which induces an isomorphism $\Br_{\Delta Q}(\Endom_\ringO (X_0)) \simeq \br_{\Delta Q}^{A} (aAa)$. The inclusion $X_0\subseteq J_QL_0$ yields $aAa\subseteq J_QA$, hence $\br_{\Delta Q}^A(aAa)=0$. We obtain $\Br_{\Delta Q}(\Endom_\ringO(X_0))=0$, a contradiction since $Q$ is a vertex of $X_0$. 

This contradiction proves that $(Q,e_Q)\leqslant \null^g(P,e_P)$. Then the $\ringO H$-module $X=X_0$ is isomorphic to a direct summand of the induced module 
$
L_g
\ \simeq\ 
\Ind_{H\cap\null^g\!P}^H \Res^{^g\!P}_{H\cap\null^g\!P} \ gU.
$
As a consequence, the source $V$ of $X$ is isomorphic to a direct summand of the $\ringO Q$-module $\Res^{^{hg}\!P}_Q hgU$ for some $h\in H$. Then we set $(R,e_R,W) = \null^{hg}(P,e_P,U)$, and the proof is complete.
\end{proof}

The application of Theorem \ref{thm:source triples and restriction} that we have in mind is as follows. Let $M$ be an indecomposable $\ringO Ge$-module, and $(P,e_P)$ be an $e$-subpair of the group $G$. Let $H$ be a subgroup of $G$ such that $PC_G(P)\leqslant H\leqslant N_G(P,e_P)$. Let $(Q,e_Q,V)$ be a source triple of an indecomposable direct summand of the $\ringO H$-module $e_P M$. If the vertex $Q$ contains the $p$-group $P$, then we have $C_G(Q)\leqslant C_G(P)\leqslant H$. Thus the theorem applies: there exists a source triple $(R,e_R,W)$ of $M$ such that $(Q,e_Q)\leqslant (R,e_R)$ and that $V$ is a direct summand of the $\ringO Q$-module $\Res^R_Q W$. 
Notice that this statement is unlikely to be true without the assumption $P\leqslant Q$.

\section{Brauer-friendly modules}
\label{sec:BFModules}

Let $P$ be a $p$-group. An endopermutation $\ringO P$-module, as defined in \cite{Dade1978}, is an $\ringO P$-module $V$ such that the endomorphism algebra $\Endom_\ringO(V)$ is $\ringO$-free and admits a $\Delta P$-stable basis. Two endopermutation $\ringO P$-modules $V$ and $W$ are said to be compatible if the direct sum $V\oplus W$ is an endopermutation $\ringO P$-module.

In this section, we fix a finite group $G$ and a block $e$ of the algebra $\ringO G$. We aim at studying certain $\ringO Ge$-modules with endopermutation sources. The following definition is derived from \cite{LinckelmannMazza}; notice however that we work over the local ring $\ringO$ rather than its residue field $k$, and that we replace the language of fusion systems with that of Brauer categories.

\begin{defn}
\label{defn:BFsource triple}
Let $(P,e_P)$ be an $e$-subpair of the group $G$, and let $V$ be an endopermutation $\ringO P$-module. We say that $V$ is fusion-stable in the group $G$ with respect to the subpair $(P,e_P)$ if the endopermutation $\ringO Q$-modules $\Res_{\phi_1} V$ and $\Res_{\phi_2} V$ are compatible for any $e$-subpair $(Q,e_Q)$ and any two arrows $\phi_1,\phi_2:(Q,e_Q)\to (P,e_P)$ in the Brauer category $\catbr(G,e)$.
\end{defn}

In practice, let $(P,e_P)$ is an $e$-subpair and $V$ an endopermutation $\ringO P$-module. If we wish to prove that $V$ is fusion-stable with respect to the subpair $(P,e_P)$, it is enough to check that the endopermutation $\ringO Q$-modules $\Res^P_QV$ and $\Res^{^g\!P}_Q gV$ are compatible, for any $e$-subpair $(Q,e_Q)$ contained in $(P,e_P)$ and any element $g\in G$ such that $(Q,e_Q)\leqslant \null^g(P,e_P)$.

To make it shorter, we say that a triple $(P,e_P,V)$ is a fusion-stable endopermutation source triple in the group $G$ if $V$ is an endopermutation $\ringO P$-module that is indecom\-po\-sable, capped (\emph{i.e.}, with vertex $P$), and fusion-stable with respect to the subpair $(P,e_P)$ of the group $G$.

\begin{defn}
We say that two fusion-stable endopermutation source triples $(P_1,e_1,V_1)$ and $(P_2,e_2,V_2)$ are compatible if the endopermutation $\ringO Q$-modules $\Res_{\phi_1} V_1$ and $\Res_{\phi_2} V_2$ are compatible for any $e$-subpair $(Q,e_Q)$ and any two arrows $\phi_1:(Q,e_Q)\to (P_1,e_1)$, $\phi_2:(Q,e_Q)\to (P_2,e_2)$ in the Brauer category $\catbr(G,e)$. 
\end{defn}

\begin{lem}
\label{lem:compatibility conjugacy induction}
Let $(P_1,e_1,V_1)$ and $(P_2,e_2,V_2)$ be two source triples in the group $G$ with respect to the block $e$.
\begin{enumerate}[(i)]
\item If $(P_1,e_1,V_1)$ is a fusion-stable endopermutation source triple, then any $G$-conjugate of $(P_1,e_1,V_1)$ is a fusion-stable endopermutation source triple. If moreover $(P_2,e_2,V_2)$ is a fusion-stable endopermutation source triple that is compatible with $(P_1,e_1,V_1)$, then it is also compatible with any $G$-conjugate of $(P_1,e_1,V_1)$.
\item
For $i=1,2$, let $(Q_i,f_i)$ be a subpair of $(P_i,e_i)$, and let $W_i$ be a capped indecomposable direct summand of the restriction $\Res^{P_i}_{Q_i} V_i$. If $(P_1,e_1,V_1)$ and $(P_2,e_2,V_2)$ are compatible fusion-stable endopermutation source triples, then $(Q_1,f_1,W_1)$ and $(Q_2,f_2,W_2)$ are compatible fusion-stable endopermutation source triples.
\item
Let $(R,e_R)$ be an $e$-subpair and $\psi_1:(P_1,e_1)\to(R,e_R)$, $\psi_2:(P_2,e_2)\to (R,e_R)$ be arrows in the Brauer category $\catbr(G,e)$. Then $(P_1,e_1,V_1)$ and $(P_2,e_2,V_2)$ are compatible fusion-stable endopermutation source triples if, and only if, the direct sum $\Ind_{\psi_1} V_1\,\oplus\, \Ind_{\psi_2} V_2$ is an endopermutation $\ringO R$-module that is fusion-stable in the group $G$ with respect to the subpair $(R,e_R)$.
\end{enumerate}
\end{lem}

\begin{proof}
The statements in (i) and (ii) are straightforward consequences of the definition, as well as the ``if'' part of (iii). The ``only if'' part of (iii) is more subtle.

With the notations of (iii), we suppose that $(P_1,e_1,V_1)$ and $(P_2,e_2,V_2)$ are compatible fusion-stable endopermutation source triples. Up to replacing these triples by conjugates, we may suppose that the $e$-subpairs $(P_1,e_1)$ and $(P_2,e_2)$ are contained in $(R,e_R)$, and that the arrows $\psi_1$ and $\psi_2$ are the inclusion maps. We follow the lines of the proof of \cite[Lemma 6.8]{Dade1978}. The Mackey formula gives an isomorphism of $\ringO\Delta R$-modules
\begin{align*}
\Homom_\ringO(\Ind_{P_1}^R V_1,\Ind_{P_2}^R V_2)
& \ \simeq \ 
\Res^{R\times R}_{\Delta R} \Ind_{P_2\times P_1}^{R\times R}
\Homom_{\ringO}(V_1,V_2)
\\
& \ \simeq \ 
\bigoplus_{x\in P_2\!\setminus\! R/P_1}
\Ind_{\Delta Q_x}^{\Delta R}
\Homom_{\ringO}(\Res_{\phi_1^{x}}V_1,\Res_{\phi_2^{x}}V_2),
\end{align*}
where $Q_x=P_2\cap\null^x\!P_1$, the morphism $\phi_1^{x}:Q_x\to P_1$ sends  $y\in Q_x$ to $^{x^{-1}\!}y\in P_1$, and $\phi_2^{x}:Q_x\to P_2$ is the inclusion map. Let $e_{Q_x}$ be the block of $\ringO C_G(Q_x)$ such that $(Q_x,e_{Q_x})\leqslant (R,e_R)$. Then the inclusions $(P_2,e_2)\leqslant (R,e_R)$ and $Q_x\leqslant P_2$ imply $(Q_x,e_{Q_x}) \leqslant (P_2,e_2)$. Since $x$ lies in $R$, we also have $^{x^{-1}}(Q_x,e_{Q_x})\leqslant (R,e_R)$ so the inclusions $(P_1,e_1)\leqslant (R,e_R)$ and $^{x^{-1}}Q_x\leqslant P_2$ imply $^{x^{-1}}(Q_x,e_{Q_x}) \leqslant (P_2,e_2)$. As a consequence, $\phi_1^{x}:(Q_x,e_{Q_x})\to (P_1,e_1)$ and $\phi_2^{x}:(Q_x,e_{Q_x})\to (P_2,e_2)$ are arrows in the Brauer category $\catbr(G,e)$.
Since the triples $(P_1,e_1,V_1)$ and $(P_2,e_2,V_2)$ are compatible, we deduce that $\Homom_\ringO(\Ind_{P_1}^R V_1,\Ind_{P_2}^R V_2)$ is a permutation $\ringO\Delta R$-module.

Similarly, the endomorphism algebras $\Endom_\ringO(\Ind_{P_1}^R V_1)$ and $\Endom_\ringO(\Ind_{P_2}^R V_2)$ are permutation $\ringO\Delta R$-module. Thus $\Ind_{P_1}^R V_1$ and $\Ind_{P_2}^R V_2$ are compatible endopermutation $\ringO R$-modules.

More generally, let $(Q,e_Q)$ be an $e$-subpair that is contained in $(R,e_R)$ and let $g\in G$ be an element such that $(Q,e_Q)\leqslant\null^g (R,e_R)$. A similar use of the Mackey formula proves that $\Homom_\ringO(\Res^R_Q \Ind_{P_1}^R V_1,\Res^{^g\!R}_Q g \Ind_{P_2}^R V_2)$ is a permutation $\ringO \Delta Q$-module. It follows that the restrictions 
\[
\Res^R_Q\Ind_{P_1}^R V_1 \ \ ;\ \ 
\Res^R_Q\Ind_{P_2}^R V_2  \ \ ;\ \ 
\Res^{^g\!R}_Q g\Ind_{P_1}^R V_1  \ \ ;\ \  
\Res^{^g\!R}_Q g\Ind_{P_2}^R V_2
\]
are compatible endopermutation $\ringO Q$-modules. So the direct sum $\Ind^R_{P_1} V_1\oplus \Ind^R_{P_2} V_2$ is an endopermutation $\ringO R$-module that is fusion-stable with respect to the subpair $(R,e_R)$.
\end{proof}

We know from Lemma \ref{lem:compatibility conjugacy induction} (i) that the notion of  compatible fusion-stable endo\-permutation source triples is invariant by conjugation in the group $G$. The source triples of an indecomposable $\ringO Ge$-module are defined up to conjugation, so the following definition is unambiguous.

\begin{defn}
We say that an $\ringO Ge$-module $M$ is Brauer-friendly if it is a direct sum of indecomposable $\ringO Ge$-modules with compatible fusion-stable endopermutation source triples. We say that two Brauer-friendly $\ringO Ge$-modules $L$ and $M$ are compatible if the direct sum $L\oplus M$ is a Brauer-friendly $\ringO Ge$-module.
\end{defn}

Let us give a few examples.
It is clear that the triples $(P,e_P,\ringO)$, where $(P,e_P)$ runs into the set of $e$-subpairs of the group $G$ and $\ringO$ is the trivial $\ringO P$-module, are compatible fusion-stable endo\-permutation source triples. Thus the $p$-permutation $\ringO Ge$-modules are Brauer-friendly and pairwise compatible.

Next, let $M$ be an indecomposable endo-$p$-permutation $\ringO Ge$-module, and $(P,e_P,V)$ be a source triple of $M$. By \cite[Theorem 1.5]{Urfer2007}, the source $V$ is an endopermutation $\ringO P$-module that is fusion-stable in the Frobenius category $\catfr(G)$ with respect to the $p$-subgroup $P$. \emph{A fortiori}, $V$ is fusion-stable in the Brauer-category $\catbr(G,e)$ with respect to the $e$-subpair $(P,e_P)$, so $M$ is a Brauer-friendly $\ringO Ge$-module. If $M$ and $N$ are compatible endo-$p$-permutation $\ringO Ge$-modules (\emph{i.e.}, the direct sum $M\oplus N$ is an endo-$p$-permutation module), then they are compatible as Brauer-friendly modules.

Conversely, let $M$ be a Brauer-friendly $\ringO Ge$-module. If $e=e_0$ is the principal block of the group $G$, then an $e$-subpair $(P,e_P)$ is uniquely determined by the $p$-subgroup $P$ of $G$. Thus it follows from \cite[Theorem 1.5]{Urfer2007} that $M$ is an endo-$p$-permutation $\ringO Ge$-module. But the fusion system of an arbitrary block is usually ``finer'' than the fusion system of the principal block, so $M$ may not be an endo-$p$-permutation module in general.

The following lemma, of which a proof is straightforward from Lemma \ref{lem:compatibility conjugacy induction} and the proof of Theo\-rem \ref{thm:source triples and restriction}, gives a more general example of Brauer-friendly module.

\begin{lem}
\label{lem:BFexample}
Let $(P,e_P)$ be an $e$-subpair, and let $V$ be an endopermutation $\ringO P$-module that is fusion-stable with respect to the subpair $(P,e_P)$. Let $i\in (\ringO Ge)^P$ be an idempotent such that $\bar e_P\br_P(i)\neq 0$. Assume that, for any subgroup $Q$ of $P$, the idempotent $\br_Q(i)$ lies in a single block of the algebra $kC_G(Q)$. Then the $\ringO Ge$-module $L=\ringO Gi\otimes_{\ringO P} V$ is Brauer-friendly.
\end{lem}

Notice that the compatibility of endopermutation modules is preserved by the reduction from the local ring $\ringO$ to the residue field $k$. For any Brauer-friendly $\ringO Ge$-module $M$, it follows that the reduction $k\otimes_\ringO M$ is a Brauer-friendly $kGe$-module. The notion is also partially compatible with the restriction to a local subgroup, as appears in the following lemma.
 
\begin{lem}
\label{lem:Brauer-friendly restriction}
Let $M$ be a Brauer-friendly $\ringO Ge$-module. Let $(P,e_P)$ be an $e$-subpair of the group $G$, and $H$ be a subgroup of $G$ such that $PC_G(P)\leqslant H\leqslant N_G(P,e_P)$. 
\begin{enumerate}[(i)]
\item The $\ringO He_P$-module $e_PM$ admits the decomposition $e_P M = L \oplus L'$, where $L$ is a Brauer-friendly $\ringO He_P$-module and $L'$ is a direct sum of indecomposable $\ringO He_P$-modules with vertices that do not contain the $p$-subgroup $P$. 
\item The restriction $\Res^H_P L$ is an endopermutation $\ringO P$-module.
\end{enumerate}
\end{lem}

\begin{proof}
The $\ringO H$-module $e_PM$ certainly admits the decomposition $e_P M = L \oplus L'$, where $L$ is a direct sum of indecomposable $\ringO H$-modules with vertices that contain $P$ and $L'$ is a direct sum of indecomposable $\ringO H$-modules with vertices that do not contain $P$. Let $X$, $X'$ be indecomposable direct summands of the $\ringO He_P$-module $L$, and let $(Q,f,W)$, $(Q',f',W')$ be respective source triples of $X$, $X'$. Since the $\ringO Ge$-module $M$ is Brauer-friendly, it follows from Theorem \ref{thm:source triples and restriction} that $(Q,f,W)$ and $(Q',f',W')$ are compatible fusion-stable source triples. Thus the $\ringO He_P$-module $L$ is Brauer-friendly .

Moreover, we have $(P,e_P)\leqslant \null^h(Q,f)$ and $(P,e_P)\leqslant \null^{h'}(Q',f')$ for any two elements $h,h'\in H$. So the restrictions $\Res^{^h\!Q}_P hW$ and $\Res^{^{h'}\!Q'}_P h'W'$ are compatible endopermutation $\ringO P$-modules. Then it follows from the Mackey formula that the restrictions $\Res^{H}_P\Ind_Q^{H} W$ and $\Res^{H}_P\Ind_{Q'}^{H} W'$ are compatible endopermutation $\ringO P$-modules. Since the restriction $\Res^{H}_P L$ is a direct sum of direct summands of modules of this kind, it is an endopermutation $\ringO P$-module.
\end{proof}

We conclude this section with a lemma that connects our notion of Brauer-friendly modules with Linckelmann's notion of modules with fusion-stable endopermutation sources over a source algebra, which appears in \cite[\S 3]{LinckelmannUnpublished}. It should be mentioned here that it was Linckelmann who initially gave us the idea to study such modules.

\begin{lem}
\label{lem:BF and source algebra}
Let $G$ be a finite group and $e$ be a block of the group $G$. Let $(D,e_D)$ be a maximal $e$-subpair and $i$ be a primitive idempotent of the algebra $(\ringO Ge)^D$ such that $\bar e_D\br_D(i)\neq 0$, \emph{i.e.}, a source idempotent. Let $A=i\ringO G i$ be the corresponding source algebra. An $\ringO Ge$-module $M$ is Brauer-friendly if, and only if, the $A$-module $iM$ is an endopermutation $\ringO D$-module.
\end{lem}

\begin{proof}
Let $M$ be an $\ringO Ge$-module, and let $M= M_1\oplus\cdots\oplus M_n$ be a Krull-Schmidt decomposition. For each integer $l\in\{1,\ldots,n\}$, let $(P_l,e_l,V_l)$ be a source triple of the indecomposable $\ringO Ge$-module $M_l$. By Lemma \ref{lem:vertex subpair and Brauer} (ii), there exists a defect group $D_l$ and a source idempotent $j_l$ of the block $e$ with respect to the defect group $D_l$ such that $P_l\leqslant D_l$ and $\bar e_l\br_{P_l}(j_l)\neq 0$, and such that $M_l$ is isomorphic to a direct summand of the $\ringO Ge$-module $\ringO Gj_l \otimes_{\ringO D_l} \Ind_{P_l}^{D_l} V_l$. Up to replacing the defect group $D_l$ and the idempotent $j_l$ by conjugates, we may assume that $D_l=D$ and $j_l=i$.
We set $W=\Ind_{P_1}^D V_1\oplus\cdots\oplus \Ind_{P_n}^D V_n$.
Then $M$ is isomorphic to a direct summand of the $\ringO G$-module $L=\ringO Gi\otimes_{\ringO D} W$, and $iM$ is isomorphic to a direct summand of the $A$-module $iL = A \otimes_{\ringO D} W$.

Suppose that the $\ringO Ge$-module $M$ is Brauer-friendly. Then, by Lemma \ref{lem:compatibility conjugacy induction} (iii), $W$ is an endopermutation $\ringO D$-module that is fusion-stable with respect to the subpair $(D,e_D)$. By \cite[Proposition 3.2 (i)]{LinckelmannUnpublished}, this implies that the $A$-module $iM$ is an endopermutation $\ringO D$-module.

Conversely, suppose that $iM$ is an endopermutation $\ringO D$-module. Let $\catf$ be the fusion system of the block $e$ with respect to the maximal subpair $(D,e_D)$. We know from \cite[\S 47]{Thévenaz1995} or \cite[\S 2]{LinckelmannUnpublished} that  the fusion system $\catf$ can be read in the $(\ringO D,\ringO D)$-bimodule structure of the source $A=i\ringO Gi$. 
Since $iM$ is an $A$-module, it follows that the endopermutation $\ringO D$-module $iM$ is fusion-stable with respect to the subpair $(D,e_D)$. For any $l\in\{1,\ldots,n\}$, the source $V_l$ is isomorphic to a direct summand of the restriction $\Res^D_{P_l} iM$. Thus, by Lemma \ref{lem:compatibility conjugacy induction} (ii), the triples $(P_l,e_l,V_l)$, $1\leqslant l\leqslant n$, are compatible fusion-stable endopermutation source triples. So $M$ is a Brauer-friendly $\ringO Ge$-module.
\end{proof}

\section{Slash functors}
\label{sec:SlashFunc}

In this section, we extend Dade's slash construction, which has been defined in \cite[Theo\-rem 4.15]{Dade1978}, to take subpairs and Brauer-friendly modules into account. The slash construction is often called deflation-restriction, \emph{e.g.} in \cite{LinckelmannUnpublished} or \cite{Thévenaz2007}. In our context, this would become \emph{deflation-truncation-restriction}. Let us briefly recall Dade's original construction.

Let $R$ be a $p$-group and $P$ be a subgroup of $R$. Let $V$ be an endopermutation $\ringO R$-module. The Brauer quotient $\Br_{\Delta P}(\Endom_\ringO(V))$ has a natural structure of $N_R(P)/P$-algebra over $k$. Moreover, there exists an endo\-per\-mutation $kN_R(P)/P$-module $V[P]$ and an isomorphism of $N_R(P)/P$-algebras $\Br_{\Delta P}(\Endom_\ringO(V))\simeq \Endom_k(V[P])$. The $kN_R(P)/P$-module $V[P]$, which is unique up to (non-unique) isomorphism, is called a $P$-slashed module relative to the $\ringO R$-module $V$. In particular, if $P=R$ and $V$ is a capped indecomposable endopermutation $\ringO P$-module, then the group $N_P(P)/P$ is trivial, and a slashed module $V[P]$ is just the $k$-vector space $k$.

If $V$ is a permutation $\ringO R$-module, then $V$ is \emph{a fortiori} an endopermutation $\ringO R$-module, and the Brauer quotient $\Br_P(V)$ is a $P$-slashed module relative to the $\ringO R$-module $V$. However, in general, the slash construction appears to be functorial in $\Endom_\ringO(V)$, but not in $V$. In a first step, we prove that this construction can actually be turned into a functor in $V$, provided that it is restricted to a suitable category of compatible endopermutation modules.

We consider a finite $p$-group $P$. If $L$ and $M$ are objects of the category $\catperm{\ringO P}$ of permutation $\ringO P$-modules, then one can derive from \cite[Lemma 3.3]{Broué1985} a natural isomorphism
\[
\Br_{\Delta P}(\Homom_\ringO (L,M))\ \tilde\longrightarrow\ \Homom_k(\Br_P(L),\Br_P(M)).
\]
More generally, let $\catm{\ringO P}$ be a full subcategory of the category $\catmod{\ringO P}$ of $\ringO P$-modules. We say that a functor $Sl:\catm{\ringO P}\to\catmod k$ is a $P$-slash functor if, for any two objects $L,M$ in the category $\catm{\ringO P}$, the map $\Homom_{\ringO P}(L,M)\to\Homom_k(Sl(L),Sl(M)), u\mapsto Sl(u),$ factors through an isomorphism 
\[
\Br_{\Delta P}(\Homom_\ringO (L,M)) \ \tilde\longrightarrow \ \Homom_k(Sl(L),Sl(M)).
\]

\begin{lem}
\label{lem:slash functor p-group}
Let $P$ be a finite $p$-group and $\catm{\ringO P}$ be a full subcategory of the category $\catmod{\ringO P}$. Assume that any two capped indecomposable direct summands of  objects in $\catm{\ringO P}$ are compatible endopermutation $\ringO P$-modules. 
\begin{enumerate}[(i)]
\item There exists a $P$-slash functor $\Sl_{P}:\catm{\ringO P}\to\catmod k$.
\item If $Sl:\catm{\ringO P}\to\catmod k$ is another $P$-slash functor, then there exists an isomorphism of functors $\Sl_P\!\tilde{\,\to}\, Sl$, and this isomorphism is unique up to scalar multiplication.
\end{enumerate}
\end{lem}

\begin{proof}
If no indecomposable direct summand of an object of the category $\catm{\ringO P}$ is capped, then any $P$-slash functor $Sl:\catm{\ringO P}\to\catmod k$ is zero and the proof of Lemma \ref{lem:slash functor p-group} is straightforward. Thus we may assume that there exists in $\catm{\ringO P}$ an object $X$ that admits a capped indecomposable direct summand $V$.

Let $M$ be an object of $\catm{\ringO P}$. We consider $\Homom_\ringO (V,M)$ as an $\ringO {\Delta P}$-module and we set $\Sl_P(M) = \Br_{\Delta P}(\Homom_\ringO (V,M))$. Let $L,M$ be two object in $\catm{\ringO P}$ and $u:L\to M$ be a morphism of $\ringO P$-modules. Then the map $\Homom_\ringO (V,u):\Homom_\ringO (V,L)\to\Homom_\ringO (V,M)$ is a morphism of $\ringO {\Delta P}$-modules, and we set $\Sl_P(u)=\Br_{\Delta P}(\Homom_\ringO (V,u)):\Sl_P(L)\to\Sl_P(M)$. This clearly defines a functor $\Sl_{P}:\catm{\ringO P}\to\catmod k$.
By construction, for any two objects $L$, $M$ in $\catm{\ringO P}$, the functor $\Sl_P$ induces a $k$-linear map
\[
\Phi^{L,M}:\Br_{\Delta P}(\Homom_\ringO (L,M)) \to 
\Homom_k(\Br_{\Delta P}(\Homom_\ringO (V,L)),\Br_{\Delta P}(\Homom_\ringO (V,M))).
\]
\hspace{\parindent}%
It follows from the assumptions of the lemma that the $\ringO P$-modules $L$ and $M$ decompose as $L=L'\oplus L^\pprime$ and $M=M'\oplus M^\pprime$, where $L'$, $M'$ are direct sums of copies of $V$ and $L^\pprime$, $M^\pprime$ are direct sums of non-capped indecomposable $\ringO P$-modules. 
This implies $\Br_{\Delta P}(\Homom_\ringO(V,L^\pprime))=\Br_{\Delta P}(\Homom_\ringO(V,M^\pprime))= 0$. Hence we may forget about $L^\pprime$, $M^\pprime$ and assume that $L=L'$, $M=M'$. Then $L$ and $M$ are finite direct sums of copies of $V$, so it is enough to prove that $\Phi^{L,M}$ is an isomorphism in the case $L=M=V$, which is trivial since $\Br_{\Delta P}(\Homom_\ringO(V,V)) \simeq k$. It follows that $\Sl_P$ is a $P$-slash functor.

We now consider another $P$-slash functor $Sl:\catm{\ringO P}\to\catmod k$. The $\ringO P$-module $X$ is an object of the category $\catm{\ringO P}$, and there is an idempotent $i\in \Endom_{\ringO P}(X)$ such that $V=iX$. 
We set $j=Sl(i)\in \Endom_k(Sl(X))$, and $W=j \,Sl(X)$.
The map $Sl^X:\Endom_\ringO (X)\to \Endom_k(Sl(X))$ induces an isomorphism $\Br_{\Delta P}(\Endom_\ringO (X)) \isom\Endom_k(Sl(X))$ that sends the idempotent $\br_P(i)$ to $j$, hence an isomorphism 
$\Br_{\Delta P}(\Endom_\ringO (V)) 
\isom
\Endom_k(W)$. It follows that $W$ is a 1-dimensional $k$-vector space. The choice of an isomorphism $\zeta : k\isom W$ brings a natural isomorphism
\[
\phi_M:\Homom_k(W,Sl(M))\ \longisom\ Sl(M), \quad M\in\catob(\catm{\ringO P}).
\]
By assumption, the $P$-slash functor $Sl$ induces an isomorphism $\Br_{\Delta P}(\Homom_\ringO (V,M)) \isom$\linebreak $\Homom_k(W,Sl(M))$, \emph{i.e.}, a natural isomorphism
\[
\psi_M:\Sl_P(M) \ \longisom \ \Homom_k(W,Sl(M)), \quad M\in\catob(\catm{\ringO P}).
\]
So we obtain a natural isomorphism $\xi=\phi\circ\psi:\Sl_P \isom Sl$. Moreover the correspondence $\zeta\leftrightarrow\xi$ is one-to-one, so $\xi$ is unique up to scalar multiplication.
\end{proof}

The second step is to extend the notion of a slash functor to take subpairs into account. In such a situation, what we call a slash functor is actually a little more than a functor.

\begin{defn}
\label{defn:slash functof subpair}
Let $G$ be a finite group, $e$ be a block of the group $G$, and $\catm{\ringO Ge}$ be a subcategory of the category $\catmod{\ringO Ge}$ of $\ringO Ge$-modules. Let $(P,e_P)$ be an $e$-subpair of the group $G$, and $H$ be a subgroup of $G$ such that $PC_G(P)\leqslant H\leqslant N_G(P,e_P)$. Write $\bar H = H/P$. A $(P,e_P)$-slash functor $Sl:\catm{\ringO Ge}\to\catmod{k\bar H\bar e_P}$ is defined by the following data:
\begin{enumerate}[\quad $\bullet$]
\item\vspace{-2mm} for each object $M$ of the category $\catm{\ringO Ge}$, a $k\bar H\bar e_P$-module $Sl(M)$;
\item\vspace{-2mm} for each pair $L,M$ of objects of the category $\catm{\ringO Ge}$, a map
\[
Sl^{L,M}:\Homom_{\ringO P}(L,M)\ \longrightarrow\ \Homom_k(Sl(L),Sl(M));
\]
\end{enumerate}
\vspace{-4mm}
such that
\begin{enumerate}[\quad $\bullet$]
\item\vspace{-2mm} $Sl^{M,M}(1_{\Endom_\ringO (M)}) = 1_{\Endom_k(Sl(M))}$ for any object $M$ of the category $\catm{\ringO Ge}$;
\item\vspace{-2mm} $Sl^{L,N}(v\circ u) = Sl^{M,N}(v)\circ Sl^{L,M}(u)$ for any three objects $L,M,N$ of the category $\catm{\ringO Ge}$ and any two morphisms of $\ringO P$-modules $u:L\to M$, $v:M\to N$;
\item\vspace{-2mm} for any two objects $L,M$ of the category $\catm{\ringO Ge}$, the map $Sl^{L,M}$ factors through an isomorphism of $k(C_G(P)\times C_G(P))\Delta H$-modules 
\[
\Br_{\Delta P}(\Homom_\ringO (e_PL,e_PM))\ \tilde\longrightarrow\  \Homom_k(Sl(L),Sl(M)).
\]
\end{enumerate}
\end{defn}

The first example of a $(P,e_P)$-slash functor is the Brauer functor $\Br_{(P,e_P)} :\catperm{\ringO Ge}\to\catmod{k\bar N_G(P,e_P)\bar e_P}$. To obtain a more general example, we need to consider a subcategory $\catm{\ringO Ge}$ of the category $\catmod{\ringO Ge}$ such that any two objects of $\catm{\ringO Ge}$ are compatible Brauer-friendly $\ringO Ge$-modules. This is what we call a Brauer-friendly category of $\ringO Ge$-modules. 

\begin{rmk} The category of all Brauer-friendly $\ringO Ge$-modules does not fit our purposes, because its objects need not be compatible with one another (unless the block $e$ has defect zero). On the contrary, there are usually various Brauer-friendly categories of $\ringO Ge$-modules. 
For instance, let $(D,e_D)$ be a maximal $e$-subpair, and $\catf = \catf_{(G,e)}(D,e_D)$ be the corresponding fusion system. Denote by $\dade(D,\catf)$ the Dade group of that fusion system, as defined in \cite{LinckelmannMazza}. Any element of that Dade group determines a capped indecomposable endopermutation $\ringO D$-module $V$ that is fusion-stable in the group $G$ with respect to the subpair $(D,e_D)$, as can be deduced from \cite[Theorem 14.2]{Thévenaz2007}. Then we let $\catm{\ringO Ge}(D,e_D,V)$ be the full subcategory of $\catmod{\ringO Ge}$ of which the objects are the direct sums of indecomposable $\ringO Ge$-modules with source triples compatible with $(D,e_D,V)$. In this way, we obtain a collection of Brauer-friendly categories, indexed by the Dade group $\dade(D,\catf)$. In particular, the trivial element of $\dade(D,\catf)$ corresponds to the category $\catperm{\ringO Ge}$ of $p$-permutation $\ringO Ge$-modules.
\end{rmk}

We can easily prove the existence and uniqueness of slash functors over such categories in the case $H=PC_G(P)$. We write $\bar C_G(P) = PC_G(P)/P$.

\begin{lem}
\label{lem:slash functor subpair centraliser}
Let $G$ be a finite group, $e$ be a block of the group $G$, and $\catm{\ringO Ge}$ be a Brauer-friendly category of $\ringO Ge$-modules. Let $(P,e_P)$ be an $e$-subpair of the group $G$.
\begin{enumerate}[(i)]
\item There exists a $(P,e_P)$-slash functor $\Sl_{(P,e_P)}:\catm{\ringO Ge}\to\catmod{k\bar C_G(P)\bar e_P}$. 
\item If $Sl:\catm{\ringO Ge}\to\catmod{k\bar C_G(P)\bar e_P}$ is another $(P,e_P)$-slash functor, then there exists an isomorphism of slash functors $\Sl_{(P,e_P)}\!\tilde{\,\to}\,Sl$, which is unique up to scalar multiplication.
\end{enumerate}
\end{lem}

\begin{proof}
Let $\catm{\ringO P}$ be the full subcategory of $\catmod{\ringO P}$ over the essential image of the truncation-restriction functor $e_P\Res^G_P$.
By Lemma \ref{lem:Brauer-friendly restriction}, the category $\catm{\ringO P}$ satisfies the assumptions of Lemma \ref{lem:slash functor p-group}. Thus there exists a $P$-slash functor $\Sl_P:\catm{\ringO P}\to\catmod k$. 

For any object $M$ in $\catm{\ringO Ge}$, we have a morphism of algebras $\Sl_P^{e_P M}:\Endom_{\ringO P}(e_PM) \to \Endom_k(\Sl_P(e_PM))$, and a natural group morphism $\iota:C_G(P) \to \Endom_{\ringO P}(e_PM)^\times$. The composition $\Sl_P^{e_P M}:C_G(P) \to \Endom_k(\Sl_P(e_PM))^\times$ makes the endomorphism ring $\Endom_k(\Sl_P(e_PM))$ a $C_G(P)$-interior algebra. As a consequence, it makes the slashed module $\Sl_P(e_PM)$ a $k\bar C_G(P)\bar e_P$-module, which we denote by $\Sl_{(P,e_P)}(M)$.

For any two objects $L,M$ in $\catm{\ringO Ge}$ and any morphism of $\ringO P$-module $u:L\to M$, we denote by $e_Pue_P:e_PL\to e_PM$ the morphism of $\ringO P$-modules induced by $u$. 
We consider $\Sl_P(e_Pue_P)$ as a $k$-linear map $\Sl_{(P,e_P)}(L)\to\Sl_{(P,e_P)}(M)$, and we denote it by $\Sl_{(P,e_P)}(u)$. 
If $u$ is a morphism of $\ringO Ge$-modules, then it follows from the functoriality of $\Sl_P$ and from the definition of the $\bar C_G(P)$-interior structures on $\Endom_k(\Sl_{(P,e_P)}(L))$ and $\Endom_k(\Sl_{(P,e_P)}(M))$ that $\Sl_{(P,e_P)}(u)$ is a morphism of $k\bar C_G(P)$-modules.
This defines a $(P,e_P)$-slash functor $\Sl_{(P,e_P)}:\catm{\ringO Ge}\to\catmod{k\bar C_G(P)\bar e_P}$. 

If $Sl:\catm{\ringO Ge}\to\catmod{k\bar C_G(P)\bar e_P}$ is another $(P,e_P)$-slash functor, then $Sl$ induces a $P$-slash functor $Sl' :\catm{\ringO P}\to \catmod k$, such that $Sl' \circ\,e_P\Res^G_P = \Res^{\bar C_G(P)}_1 \circ\, Sl$. By Lemma \ref{lem:slash functor p-group}, there exists an isomorphism of functors $\xi':\Sl_P\to Sl'$. For any object $M$ of the category $\catm{\ringO Ge}$, the isomorphism of $k$-vector spaces $\xi'_{e_PM}:\Sl_P(e_PM)\to Sl'(e_P M)$ may be seen as an isomorphism of $k\bar C_G(P)$-module $\Sl_{(P,e_P)}(M)\to Sl(M)$, which we denote by $\xi_M$. By construction, for any two objects $L,M$ in $\catm{\ringO Ge}$ and any morphism of $\ringO P$-module $u:L\to M$, there is a commutative diagram
\[
\xymatrix{
\Sl_{(P,e_P)}(L) \ar[rr]^{\xi_L} \ar[d]_{\Sl_{(P,e_P)}^{L,M}(u)} && 
Sl(L) \ar[d]^{Sl^{L,M}(u)} \\
\Sl_{(P,e_P)}(M) \ar[rr]^{\xi_M} &&
Sl(M)
}
\]
Thus $\xi:\Sl_{(P,e_P)}\to Sl$ is what we would like to call an isomorphism of $(P,e_P)$-slash functors. The uniqueness of $\xi$ up to scalar multiplication follows from the similar uniqueness of $\xi'$.
\end{proof}

The third and last step is to consider a $(P,e_P)$-slash functor with a codomain such as the category $\catmod{k\bar H\bar e_P}$, where $H$ is any subgroup of $G$ with $PC_G(P)\leqslant H\leqslant N_G(P,e_P)$. The existence of such a slash functor follows from a deep result proven by Puig in \cite{Puig1986}. Unfortunately, we lose the uniqueness of a slash functor up to isomorphism. We need to explain how a $(P,e_P)$-slash functor $Sl:\catm{\ringO Ge}\to\catmod{k\bar H\bar e_P}$ may be twisted by a linear character $\chi:\bar H/\bar C_G(P) \to k^\times$.

If $M$ is an object of the category $\catm{\ringO Ge}$, then $Sl(M)$ is a $k\bar H\bar e_P$-module. We set $\chi_*Sl(M) = Sl(M)$ as a $k$-vector space, and we endow $\chi_*Sl(M)$ with the action $\cdot_\chi$ of the group $\bar H$ defined by $h \cdot_\chi m = h\cdot \chi(h)m$ for any $h\in \bar H$ and $m\in Sl(M)$, where the single dot stands for the preexisting action of the group $\bar H$ on the $k$-vector space $Sl(M)$. If $L,M$ are two objects of $\catm{\ringO Ge}$ and $u:L\to M$ is a morphism of $\ringO P$-modules, then we set $\chi_*Sl^{L,M}(u) = Sl^{L,M}(u)$, considered as a $k$-linear map $\chi_*Sl(L)\to\chi_*Sl(M)$. This defines another $(P,e_P)$-slash functor $\chi_*Sl:\catm{\ringO Ge}\to\catmod{k\bar H\bar e_P}$. Notice that the slash functors $Sl$ and $\chi_*Sl$ might be isomorphic.

\begin{thm}
\label{thm:slash functor subpair}
Let $G$ be a finite group, $e$ be a block of the group $G$, and $\catm{\ringO Ge}$ be a Brauer-friendly category of $\ringO Ge$-modules. Let $(P,e_P)$ be an $e$-subpair of the group $G$, and $H$ be a subgroup of $G$ such that $PC_G(P)\leqslant H\leqslant N_G(P,e_P)$.
\begin{enumerate}[(i)]
\item There exists a $(P,e_P)$-slash functor $Sl_{(P,e_P)}:\catm{\ringO Ge}\to\catmod{k\bar H\bar e_P}$. 
\item If $Sl:\catm{\ringO Ge}\to\catmod{k\bar H\bar e_P}$ is another $(P,e_P)$-slash functor, then there exists a 
linear character $\chi:\bar H/\bar C_G(P)\to k^\times$ and an isomorphism of slash functors \mbox{$\chi_* Sl_{(P,e_P)} \!\tilde{\,\to}\, Sl$.} 
\end{enumerate}
\end{thm}

\begin{proof}
We know from Lemma \ref{lem:slash functor subpair centraliser} (i) that there exists a $(P,e_P)$-slash functor $\Sl_{(P,e_P)}:\catm{\ringO Ge}\to\catmod{k\bar C_G(P)\bar e_P}$. We may assume that this slash functor is nonzero, \emph{i.e.}, that there exists an object $X$ of $\catm{\ringO Ge}$ such that $\Sl_{(P,e_P)}(X)\neq 0$.

Let $L$ and $M$ be two objects of the category $\catm{\ringO Ge}$. The $\ringO$-module $\Homom_\ringO(e_PL,e_PM)$ has a natural structure of $\ringO(H\times H)$-module, so the Brauer quotient $\Br_{\Delta P}(\Homom_\ringO(e_PL,e_PM))$ admits a natural structure of $k(C_G(P)\times C_G(P))\Delta N_G(P,e_P)$-module. If $L=M$, the Brauer quotient $\Br_{\Delta P}(\Endom_\ringO(e_PM)) \simeq \Endom_k(\Sl_{(P,e_P)}(M))$ is more precisely a $C_G(P)$-interior $N_G(P,e_P)$-algebra, as defined in \cite{Puig2002}. 

The main result of \cite{Puig1986} implies that $\Endom_k(\Sl_{(P,e_P)}(M))$ may be extended (non-uniquely) to an $H$-interior algebra. This defines an extension of the slashed module $\Sl_{(P,e_P)}(M)$ to a $k\bar H$-module, which is unique up to twisting by a linear character $\chi:\bar H/\bar C_G(P)\to k^\times$. Let us choose, once and for all, such an extension for the slashed module $\Sl_{(P,e_P)}(X)$.

Then, for any object $M$ of $\catm{\ringO Ge}$, the slashed module $\Sl_{(P,e_P)}(M)$ may admit several extensions to a $k\bar H$-module. Each one of them defines a  structure of $k(\bar H\times \bar H)$-module on $\Homom_k(\Sl_{(P,e_P)}(X),\Sl_{(P,e_P)}(M))$. By definition, the slash functor $\Sl_{(P,e_P)}$ induces an isomorphism of $k(C_G(P)\times C_G(P))$-modules
\[
\Br_{\Delta P}(\Homom_\ringO (e_PX,e_PM))\ \tilde\longrightarrow\  \Homom_k(\Sl_{(P,e_P)}(X),\Sl_{(P,e_P)}(M)).
\]
Since the slashed module $\Sl_{(P,e_P)}(X)$ is nonzero, there is only one extension of $\Sl_{(P,e_P)}(M)$ to a $k\bar H$-module such that the above map is a morphism of $k(C_G(P)\times C_G(P))\Delta H$-module. We denote this $k\bar He_P$-module by $Sl_{(P,e_P)}(M)$, in italics. This defines a $(P,e_P)$-slash functor
\[
Sl_{(P,e_P)}:\catm{\ringO Ge}\ \longrightarrow \ \catmod{k\bar H\bar e_P}.
\] 
\hspace{\parindent}%
Let $Sl:\catm{\ringO Ge}\ \longrightarrow \ \catmod{k\bar H\bar e_P}$ be another $(P,e_P)$-slash functor. Then, by Lemma \ref{lem:slash functor subpair centraliser} (ii), the restriction $\Res^{\,\bar H}_{\bar C_G(P)}Sl$ is isomorphic to the slash functor $\Sl_{(P,e_P)}$. By transport of structure, we may suppose that $\Res^{\,\bar H}_{\bar C_G(P)}Sl = \Sl_{(P,e_P)}$, \emph{i.e.}, the only difference between $Sl_{(P,e_P)}$ and $Sl$ is the choice of an extension of the slashed module $\Sl_{(P,e_P)}(X)$ to a $k\bar H$-module. Up to twisting $Sl_{(P,e_P)}$ by a linear character, we can obtain $Sl_{(P,e_P)}(X)=Sl(X)$, which implies $Sl_{(P,e_P)}=Sl$. This completes the proof of the theorem.
\end{proof}

In general, there are several $(P,e_P)$-slash functors on a given Brauer-friendly category, none of which can be singled out. However, in the case of the category $\catperm{\ringO Ge}$ of $p$-permutation modules, there is a canonical choice : the Brauer functor~$\Br_{(P,e_P)}$.

\section{The essential image of a slash functor}
\label{sec:Slash}

In this section, we prove that the essential image of a slash functor is, as expected, a Brauer-friendly category. This enables us to compose slash functors, and to prove that such a composition results in another slash functor. These results are prepared by two technical lemmas, which will also be used in the next section.

\begin{lem}
\label{lem:slash induction}
Let $H$ be a finite group and $P$ be a normal $p$-subgroup of $H$. Let $R$ be a $p$-subgroup of $H$ that contains $P$, and $V$ be an $\ringO R$-module. Suppose that the restriction $\Res^H_P\Ind_R^H V$ is an endopermutation $\ringO P$-module. The $C_H(P)$-interior $H$-algebra $\Br_{\Delta P}(\Ind_R^H \Endom_\ringO (V))$ admits a (non-unique) extension to an $H$-interior algebra. Once this extension has been chosen, there is a natural isomorphism of $H$-interior algebras
\[
\Phi : \Ind_R^H \Br_{\Delta P}(\Endom_k(V)) \ \longisom \ \Br_{\Delta P}(\Ind_R^H \Endom_\ringO(V)),
\]
where we write $\Ind_R^H$ for the induction of interior algebras (see \cite[\S 16]{Thévenaz1995}). 
\end{lem}

\begin{proof} 
Let us write $S=\Endom_\ringO(V)$, an $R$-interior matrix algebra over $\ringO$ that can also be seen as a $\ringO(R\times R)$-module. We may assume that $\Br_{\Delta P}(S)\neq 0$, since the conclusion of the lemma is trivial otherwise.
On the one hand, the $H$-interior algebra $A=\Ind_R^H S$ is defined, as an $\ringO(H\times H)$-module, by
\[
A \ = \ \ringO H \otimes_{\ringO R} S \otimes_{\ringO R} \ringO H.
\]
The map $\phi:S\to A, s\mapsto 1_H\otimes s\otimes 1_H$ is an embedding of algebras. It induces an isomorphism of $R$-algebras $S\isom \alpha A\alpha $, where $\alpha =\phi(1)$ is an idempotent of the algebra $A^R$. 
By \cite{Puig1986}, the $C_H(P)$-interior $H$-algebra $A_P=\Br_{\Delta P}(A)$ admits an extension to a $H$-interior algebra, which makes $A_P$ a $k(H\times H)$-module. Consider the idempotent $\alpha_P=\br_P(\alpha)\in (A_P)^R$, and the $R$-interior subalgebra $\alpha_PA_P\alpha_P$.
The decomposition $1_{A_P}=\sum_{g\in H/R} \null^g\alpha_P$ of the unity into mutually orthogonal idempotents brings a decomposition of the $k(R\times R)$-module $A_P$:
\[
A_P 
\ =\ 
\bigoplus_{g,h\in H/R} \null^g\alpha_P \ A_P \ \null^h\alpha_P
\ =\ 
\bigoplus_{g,h\in H/R} g(\alpha_PA_P \alpha_P)h^{-1}.
\]
On the other hand,  the Brauer quotient $S_P = \Br_{\Delta P}(S)$ is a matrix algebra over $k$, with a natural $R$-interior structure that also makes it a $k(R\times R)$-module. The $H$-interior algebra $B=\Ind_R^H S_P$ is defined, as a $k(H\times H)$-module, by
\[
B \ = \ kH \otimes_{kR} S_P \otimes_{kR} kH.
\]
The embedding $\psi:S_P\to B, s\mapsto 1_H\otimes s\otimes 1_H$ induces an isomorphism of $R$-algebras $S_P\isom \beta B\beta$, where $\beta =\psi(1)\in B^R$. 
The decomposition $1_B=\sum_{g\in H/R} \null^g\beta$ brings a decomposition of the $k(R\times R)$-module $B$:
\[
B = \bigoplus_{g,h\in H/R} \null^g\beta\  B\ \null^h\beta   = \bigoplus_{g,h\in H/R} g(\beta  B\beta  )h^{-1}.
\]
The Brauer functor $\Br_{\Delta P}$ sends the map $\phi:S\to A$ to a morphism of algebras $\phi_P:S_P\to A_P$, which restricts to an isomorphism of $R$-algebras $\phi'_P:S_P\isom \alpha_PA_P\alpha_P$. By uniqueness of the $R$-interior structure on $S_P$, $\phi'_P$ is an isomorphism of $R$-interior algebras. Thus $\phi_P:S_P\to A_P$ is a morphism of $k(R\times R)$-modules. By the universal property of induced modules, there exists a unique morphism of $k(H\times H)$-modules $\Phi:B\to A_P$ such that $\Phi\circ \psi=\phi_P$.

By construction, we have $\Phi(\beta  )=\alpha_P$ and $\Phi$ induces an isomorphism of $R$-interior algebras $\beta  B\beta  \isom \alpha_PA_P\alpha_P$. Since $\Phi$ is a morphism of $k(H\times H)$-modules, it follows from the above decompositions of $B$ and $A_P$ that $\Phi$ is an isomorphism of $k(R\times R)$-modules. Then the definition of induced interior algebras implies that $\Phi$ is an isomorphism of $H$-interior algebras.
\end{proof}

\begin{lem}
\label{lem:slash induction with idempotent}
With the assumptions of Lemma \ref{lem:slash induction}, let $i$ be an idempotent of the algebra $(\ringO H)^R$. Write $i_P=\br_P(i)$. The isomorphism $\Phi$ induces an isomorphism of $H$-interior algebras
\[
\Phi_i : kHi_P \otimes_{kR} \Br_{\Delta P}(\Endom_k(V)) \otimes_{kR} i_P kH
\ \longisom \ 
\Br_{\Delta P}(\ringO Hi \otimes_{\ringO R} \Endom_k(V) \otimes_{\ringO R} i\ringO H).
\]
\end{lem}

\begin{proof}
With the notations of the proof of Lemma \ref{lem:slash induction}, we consider the idempotent $u = \Tr_R^H(i\alpha i)$ of the algebra $A^H$, and the idempotent $u_P=\br_P(u)=\Tr_R^H(i_P\alpha_Pi_P)$ of the algebra $(A_P)^H$. A~direct computation in the induced algebra $A$ yields
\[
uAu = \ringO Hi\otimes_{\ringO R} S \otimes_{\ringO R} i\ringO H
, \quad \text{ so that } \quad
u_PA_Pu_P \simeq \Br_{\Delta P}(\ringO Hi\otimes_{\ringO R} S \otimes_{\ringO R} i\ringO H).
\]
Then we consider the idempotent $v = \Tr_R^H(i_P\otimes 1_{\Br_{\Delta P}(S)} \otimes i_P) = \Tr_R^H(i_P\beta i_P)$ of the algebra $B^H$. The same computation brings
\[
vBv = kHi_P\otimes_{kR} \Br_{\Delta P}(S) \otimes_{kR} i_P kH.
\]
Since the morphism $\Phi$ sends $\beta$ to $\alpha_P$ and commutes with the relative trace map, we obtain $\Phi(v)=u_P$. So $\Phi$ induces an isomorphism of $H$-interior algebras $\Phi_i: vBv\to u_PA_Pu_P$.
\end{proof}

We can now prove the main result of this section.

\begin{lem}
\label{lem:slash functor subpair image}
Let $G$ be a finite group, $e$ be a block of the algebra $\ringO G$, and $\catm{\ringO Ge}$ be a Brauer-friendly category of $\ringO Ge$-modules. Let $(P,e_P)$ be an $e$-subpair of the group $G$, $H$ be a subgroup of $G$ such that $PC_G(P)\leqslant H\leqslant N_G(P,e_P)$, and
\[
Sl_{(P,e_P)} \,:\, \catm{\ringO Ge} \ \to\ \catmod{k\bar H\bar e_P}
\]
be a $(P,e_P)$-slash functor. Then there exists a Brauer-friendly category $\catm{k\bar H\bar e_P}$ of $k\bar H\bar e_P$-modules that contains the essential image of $Sl_{(P,e_P)}$.
\end{lem}

\begin{proof}
We let $\mathcal S$ be the set of source triples $(Q,e_Q,V)$ of the group $G$ for which there exists a source triple $(R,e_R,W)$ of an indecomposable direct summand of an object of $\catm{\ringO Ge}$ such that $(P,e_P)\trianglelefteq (Q,e_Q)\leqslant (R,e_R)$  and $V$ is a direct summand of the $\ringO Q$-module $\Res^{R}_Q W$.
We let $\bar{\mathcal S}$ be the set of triples $(\bar Q,\bar e_Q,\bar V)$ for which there exists an element $(Q,e_Q,V)$ of the set $\mathcal S$ such that $\bar Q = Q/P$, $\bar e_Q = \br_P(e_Q)$, and the $k\bar Q$-module $\bar V$ is a capped indecomposable direct summand of a $P$-slashed module $V[P]$. We denote by $\catm{\ringO He_P}$ (\emph{resp.} $\catm{k\bar H\bar e_P}$) the full subcategory of $\catmod{\ringO He_P}$ (\emph{resp.} $\catmod{k\bar H\bar e_P}$) of which the objects are the direct sums of indecomposable modules with source triples in $\mathcal S$ (\emph{resp.} $\bar{\mathcal S}$). These are Brauer-friendly categories.

We know from Lemma \ref{thm:slash functor subpair} that there exists a $(P,e_P)$-slash functor $Sl'_{(P,e_P)}:\catm{\ringO He_P}\to\catmod{k\bar H\bar e_P}$. Let $M$ be an object of the category $\catm{\ringO Ge}$. By Lemma \ref{lem:Brauer-friendly restriction}, the $\ringO He_P$-module $e_PM$ admits the decomposition $e_P M = L \oplus L'$, where $L$ is a Brauer-friendly $\ringO He_P$-module and $L'$ is a direct sum of indecomposable $\ringO He_P$-modules with vertices that do not contain the $p$-subgroup $P$. It follows from Lemma \ref{thm:slash functor subpair} (ii) that the slashed modules $Sl_{(P,e_P)}(M)$ and $Sl'_{(P,e_P)}(L)$ are isomorphic, up to twisting by a linear character of the group $\bar H/\bar C_G(P)$. Notice that such a twist preserves the source triples of an indecomposable $k\bar H\bar e_P$-module.

Let $(\bar Q,\bar e_Q,\bar V)$ be a source triple of an indecomposable direct summand of  $Sl_{(P,e_P)}(M)$, hence of $Sl'_{(P,e_P)}(L)$. To study this triple, we may first suppose that the $\ringO He_P$-module $L$ is indecomposable. Then we may suppose, as in the proof of Theorem \ref{thm:source triples and restriction}, that $L=\ringO Hi\otimes_{\ringO R} W$, with $(R,e_R,W)$ a source triple of the group $H$ that lies in the set $\mathcal S$, and $i$ a primitive idempotent of the algebra $(\ringO H)^R$ such that $\bar e_R\br_R(i)\neq 0$. With this assumption, we know from Lemma \ref{lem:slash induction with idempotent} that the slashed module $Sl'_{(P,e_P)}(L)$ is isomorphic to the $k\bar H\bar e_P$-module $L_P=kH\bar \br_P(i)\otimes_{kQ/P} W[P]$.
Then we deduce from Theorem \ref{thm:source triples and restriction} that the subpair $(\bar R,\bar e_R)$ is contained in $(Q/P,\br_P(e_Q))$, and that $W$ is isomorphic to a direct summand of  the $k\bar R$-module $\Res^{Q/P}_{\,\bar R}V[P]$. Thus $Sl_{(P,e_P)}(M)$ lies in $\catm{k\bar H \bar e_P}$.
\end{proof}

Lemma \ref{lem:slash functor subpair image} enables us to study the transitivity of slash functors.

\begin{lem}
\label{lem:transitivity of the slash construction}
Let $G$ be a finite group, $e$ be a block of the group $G$, and $\catm{\ringO Ge}$ be a Brauer-friendly category of $\ringO Ge$-modules. 
\begin{enumerate}[(i)]
\item
Let $(P,e_P)\,\triangleleft\, (Q,e_Q)$ be $e$-subpairs of the group $G$. 
Let 
$
Sl_{(P,e_P)}:\catm{\ringO Ge}\to\catmod{k\bar N_G(P,e_P)\bar e_P}
$
be a $(P,e_P)$-slash functor, and $\catm{k\bar N_G(P,e_P)\bar e_P}$ be a Brauer-friendly category of $k\bar N_G(P,e_P)\bar e_P$-modules that contains the essential image of $Sl_{(P,e_P)}$. Let 
$
Sl_{(Q/P,\bar e_Q)}:\catm{k\bar N_G(P,e_P)\bar e_P}\to\catmod{k\bar N_G(P,Q,e_Q)\bar e_Q}
$
be an $(Q/P,\bar e_Q)$-slash functor. Then the composition 
\[
Sl_{(Q/P,\bar e_Q)} \circ Sl_{(P,e_P)} \,:\,
\catm{\ringO Ge}\ \to\ \catmod{k\bar N_G(P,Q,e_Q)\bar e_Q}
\]
is an $(Q,e_Q)$-slash functor.
\item Let $(P,e_P)$ be an $e$-subpair, and
$
Sl_{(P,e_P)}:\catm{\ringO Ge}\to\catmod{k\bar N_G(P,e_P)\bar e_P}
$
be a $(P,e_P)$-slash functor. For an element $g$ of the group $G$, let $g_\star: \catmod{k\bar N_G(P,e_P)\bar e_P}\to \catmod{k\bar N_G(\null^gP,\null^g e_P)\null ^g\bar e_P}$ stand for the ``twist by $g$''.
Then the composition 
\[
g_\star \circ Sl_{(P,e_P)} \,:\,
\catm{\ringO Ge}\ \to\ \catmod{k\bar N_G(\null^gP,\null^g e_P)\null ^g\bar e_P}
\]
is a $\,^g(P,e_P)$-slash functor.
\end{enumerate}
\end{lem}

\begin{proof} In order to prove (i), we may assume that $G=N_G(P,Q,e_Q)$ and $e=e_P=e_Q$. We fix two compatible Brauer-friendly $\ringO Ge$-modules $L$ and $M$. Thus $\Homom_\ringO(L,M)$ is a $p$-permutation $\ringO \Delta Q$-module, and we know from \cite[\S 2.2]{Rouquier1998} that the natural map
\[
\Br_{\Delta Q}(\Homom_\ringO(L,M))
\ \to\ 
\Br_{\Delta Q/P}\circ\Br_{\Delta P}(\Homom_\ringO(L,M))
\]
is an isomorphism. This proves (i); the proof of (ii) is similar.
\end{proof}

\section{A parametrisation of the indecomposable Brauer-friendly mo\-dules}
\label{sec:ParamIndBF}

In this section, we prove that an indecomposable Brauer-friendly $\ringO Ge$-module $X$ is characte\-rized, up to isomorphism, by a conjugacy class of quadruples of the form $(P,e_P,V,\bar X)$, where $(P,e_P,V)$ is a fusion-stable endopermutation source triple of the group $G$, and $\bar X$ is a projective indecomposable $k\bar N_G(P,e_P)\bar e_P$-module. The precise statement is the following.

\begin{thm}
\label{thm:BF Puig correspondence}
Let $G$ be a finite group and $e$ be a block of the group $G$.
Let $(P,e_P,V)$ be a fusion-stable endopermutation source triple of the group $G$ with respect to the block $e$. Let $\catm{\ringO Ge}$ be a Brauer-friendly category of $\ringO Ge$-modules that is ``big enough'', \emph{i.e.}, such that any finite direct sum of indecomposable $\ringO Ge$-modules with source triple $(P,e_P,V)$ is an object of $\catm{\ringO Ge}$.
Let 
\[
Sl_{(P,e_P)}:\catm{\ringO Ge}\to \catmod{k\bar N_G(P,e_P)\bar e_P}
\]
be a $(P,e_P)$-slash functor.
Then the mapping $X\mapsto Sl_{(P,e_P)}(X)$ induces a one-to-one corres\-pondence between the isomorphism classes of indecomposable $\ringO Ge$-modules with source triple $(P,e_P,V)$ and the isomorphism classes of projective indecomposable $k\bar N_G(P,e_P)\bar e_P$-modules.
\end{thm}

\begin{proof}
We write $H=N_G(P,e_P)$. There exists a Brauer-friendly category $\catm{\ringO He_P}$ of $\ringO He_P$-modules that contains any finite direct sum of indecomposable $\ringO He_P$-modules with source triple $(P,e_P,V)$, and there exists a $(P,e_P)$-slash functor $Sl'_{(P,e_P)}:\catm{\ringO He_P}\to \catmod{k\bar H\bar e_P}$.

Let $X$ be an indecomposable $\ringO Ge$-module with source triple $(P,e_P,V)$. Let $X'$ be its Green correspondent, an object of the category $\catm{\ringO He_P}$. Up to twisting the slash functor $Sl'_{(P,e_P)}$ by a linear character of the group $\bar H/\bar C_G(P)$, we may assume that the slashed modules $Sl_{(P,e_P)}(X)$ and $Sl_{(P,e_P)}(X')$ are isomorphic. Once this assumption has been made for $X$, the same is necessarily true for any other indecomposable $\ringO Ge$-module with source triple $(P,e_P,V)$.

By definition of a source triple, $X'$ is isomorphic to a direct summand of the $\ringO He_P$-module $L=\ringO He_P\otimes_{\ringO P} V$. A $P$-slashed module relative to the capped indecomposable endopermutation $\ringO P$-module $V$ is the $k$-vector space $k$. Thus, by Lemma \ref{lem:slash induction with idempotent}, there is an isomorphism of $k\bar H\bar e_P$-modules
\[
Sl'_{(P,e_P)}(L) \ \simeq\ kH\bar e_P\otimes_{kP} k \ \simeq\ k\bar H\bar e_P.
\]
\hspace{\parindent}Since $L$ is an endo-$p$-permutation $\ringO H$-module, the Brauer map $\br_{\Delta P}^{\Endom_\ringO(L)}$ induces an epimorphism $[\Endom_\ringO(L)]^H \twoheadrightarrow [\Br_{\Delta P}(\Endom_\ringO(L))]^H$. In other words, the slash functor $Sl'_{(P,e_P)}$ induces an epimorphism
\[
\beta:\Endom_{\ringO H}(L) \ \twoheadrightarrow\ \Endom_{k H}(Sl'_{(P,e_P)}(L))
\]
By a classical result about the lifting of idempotents such as \cite[Theorem 3.1]{Thévenaz1995}, the epimorphism $\beta$ induces a one-to-one correspondence between the conjugacy classes of primitive idempotents of the algebra $\Endom_{\ringO H}(L)$ and the conjugacy classes of primitive idempotents of $\Endom_{k H}(Sl'_{(P,e_P)}(L))$. In other words, the mapping $X'\mapsto Sl'_{(P,e_P)}(X')$  induces a one-to-one correspondence between the isomorphism classes of indecomposable direct summands of the $\ringO H$-module $L$ and the isomorphism classes of indecomposable direct summands of the $k\bar H$-module $Sl'_{(P,e_P)}(L)$, \emph{i.e.}, between the isomorphism classes of indecomposable $\ringO He_P$-modules with source triple $(P,e_P,V)$ and the isomorphism classes of indecomposable projective $k\bar H\bar e_P$-modules. Then the above remark concerning Green correspondents completes the proof.
\end{proof}

The one-to-one correspondence of Theorem \ref{thm:BF Puig correspondence} may be seen as an instance of the Puig corres\-pondence defined in \cite{Puig1988construction}. An indecomposable Brauer-friendly $\ringO Ge$-module $X$ is characterised by a quadruple $((P,e_P),V,\bar X)$, as described at the beginning of this section. The correspondence $X\leftrightarrow \bar X$ is usually not canonical; it depends on the isomorphism class of the slash functor $Sl_{(P,e_P)}$. This is consistent with what Thévenaz explains in \cite[before Example 26.5]{Thévenaz1995}.

\end{document}